\documentclass{amsart}
\usepackage{amssymb}

\usepackage[utf8]{inputenc}
\usepackage[T1]{fontenc}
\usepackage{mathpazo,courier}
\usepackage[scaled]{helvet}

\usepackage[usenames,dvipsnames]{xcolor}
\definecolor{darkblue}{rgb}{0.0, 0.0, 0.55}
\usepackage[pagebackref,colorlinks,linkcolor=BrickRed,citecolor=OliveGreen,urlcolor=darkblue,hypertexnames=true]{hyperref}

\usepackage{enumerate}

\newcommand\N{\mathbb N}

\newcommand\Q{\mathbb Q}
\newcommand\R{\mathbb R}
\newcommand\C{\mathbb C}

\usepackage{ushort}
\newcommand\x{\ushort X}

\newcommand\g{\ushort g}
\newcommand\h{\ushort h}

\usepackage{bbm} 
\newcommand\ii{\mathbbm i}

\newcommand\al\alpha
\newcommand\la\lambda
\newcommand\de\delta
\newcommand\ep\varepsilon
\newcommand\ph\varphi
\newcommand\si\sigma
\newcommand\ta\tau

\newcommand\dotcup{\mathbin{\dot\cup}}

\DeclareMathOperator\hess{Hess}
\DeclareMathOperator\conv{conv}
\DeclareMathOperator\sgn{sgn}
\DeclareMathOperator\id{I} 

\theoremstyle{definition}
\newtheorem{thm}{Theorem}[section]
\newtheorem{main}[thm]{Main Theorem}

\newtheorem{pro}[thm]{Proposition}
\newtheorem{rem}[thm]{Remark}
\newtheorem{ex}[thm]{Example}
\newtheorem{lem}[thm]{Lemma}
\newtheorem{df}[thm]{Definition}
\newtheorem{notation}[thm]{Notation}

\title[Lasserre relaxations for convex semialgebraic sets]
{On the exactness of Lasserre relaxations for\\compact convex basic closed semialgebraic sets}

\author[T.-L. Kriel]{Tom-Lukas Kriel}
\address{Fachbereich Mathematik und Statistik, Universität Konstanz, 78457 Konstanz, Germany}
\email{tom-lukas.kriel@uni-konstanz.de}

\author[M. Schweighofer]{Markus Schweighofer}
\address{Fachbereich Mathematik und Statistik, Universität Konstanz, 78457 Konstanz, Germany}
\email{markus.schweighofer@uni-konstanz.de}


\date{January 29, 2018}


\begin{document}
\begin{abstract}
Consider a finite system of non-strict real polynomial inequalities
and suppose its solution set $S\subseteq\R^n$ is convex, has nonempty interior and
is compact. Suppose that the system satisfies the Archimedean condition,
which is slightly stronger than the compactness of $S$.
Suppose that each defining polynomial satisfies a 
second order strict quasiconcavity condition where it vanishes on
$S$ (which is very natural because of the convexity of $S$) or its
Hessian has a certain matrix sums of squares certificate
for negative-semidefiniteness on $S$ (fulfilled trivially by linear polynomials).
Then we show that the system possesses an exact Lasserre relaxation.

In their seminal work of 2009, Helton and Nie showed under the same conditions
that $S$ is the projection of a spectrahedron, i.e., it has a semidefinite representation.
The semidefinite representation used by Helton and Nie arises from glueing together
Lasserre relaxations of many small pieces obtained in a non-constructive way.
By refining and varying their approach, we show that we can simply take a Lasserre relaxation
of the original system itself. Such a result was provided by Helton and Nie with much more
machinery only under very technical conditions and after changing the description of $S$.
\end{abstract}

\maketitle


\section{Introduction}

\noindent
Throughout the article, $\N$ and $\N_0$ denote the set of positive and nonnegative integers,
respectively. We fix $n\in\N_0$ and denote by $\x:=(X_1,\dots,X_n)$ a tuple of $n$ variables.
We denote by $\R[\x]:=\R[X_1,\dots,X_n]$ the polynomial ring in these variables over $\R$.
For $\al\in\N_0^n$, we denote $|\al|:=\al_1+\ldots+\al_n$ and
$\x^\al:=X_1^{\al_1}\dotsm X_n^{\al_n}$. For $p=\sum_\al a_\al\x^\al\in\R[\x]$
with all $a_\al\in \R$, the degree of $p$ is defined as $\deg p:=\max\{|\al|\mid a_\al\ne0\}$ if $p\ne0$
and $\deg p:=-\infty$ if $p=0$. For each $d\in\R$, we consider the real vector space
\[\R[\x]_d:=\{p\in\R[\x]\mid\deg p\le d\}\]
of all polynomials of degree at most $d$. We admit here real numbers $d$ for technical reasons but note that $\R[\x]_d=\R[\x]_{\lfloor d\rfloor}$ for all $d\in\R$ and $\R[\x]_d=\{0\}$ for all $d<0$. Occasionally, we will need the real polynomial ring in one variable as an auxiliary tool, and we will denote it by $\R[T]$.
We will denote the $n\times n$ identity matrix by $I_n$.

\bigskip\noindent
For a tuple $\g:=(g_1,\dots,g_m)\in\R[\x]^m$ of $m$ polynomials, the set
\[S(\g):=\{x\in\R^n\mid g_1(x)\ge0,\dots,g_m(x)\ge0\}\]
is called a \emph{basic closed semialgebraic set}
\cite[Def. 2.1.1]{pd}. Boolean combinations of such sets are called
\emph{semialgebraic sets} \cite[Def. 2.1.4]{pd}.
The finiteness theorem from real algebraic geometry says that every closed semialgebraic set is a finite union of basic closed ones \cite[Thm. 2.4.1]{pd}.
In general, it is hard
to answer questions about the geometry $S(\g)$ from its description $\g$. This is of course due to the
nonlinear monomials $\x^\al$ with $|\al|\ge2$ that might appear in $\g$. An extremely naive idea
would be to replace each such nonlinear monomial $\x^\al$ in $\g$ by a new variable $Y_\al$.
This would lead to a system of $m$ \emph{linear} inequalities
whose solution set is a (closed convex) polyhedron in a higher-dimensional space.
The projection of this polyhedron to the
$\x$-space $\R^n$ contains $S(\g)$ but will very often just be the whole of $\R^n$
and thus be of no help.

\bigskip\noindent
This idea becomes however less naive if we add a bunch of redundant
inequalities before the linearization.
For example, we could add certain inequalities of the form $p^2(x)\ge0$ or $(p^2g_i)(x)\ge0$
with $p\in\R[\x]$. If we choose finitely many such inequalities in a clever
way and then linearize as above,
we will get a polyhedron in a higher-dimensional space whose projection to $\x$-space $\R^n$
might enclose $S(\g)$ more tightly. Unless $S(\g)$ happens to be a polyhedron, this projection
can however still not equal $S(\g)$ since projections of polyhedra are again polyhedra
(see \cite[Subsection 12.2]{scr} for a textbook reference).

\bigskip\noindent
The idea of Lasserre was therefore to add the whole (infinite)
family of all
redundant inequalities of the form $p^2(x)\ge0$ or $(p^2g_i)(x)\ge0$ with $p\in\R[\x]$ 
before the linearization \cite{l1,l2}. To get something that is useful in practice
(for example, one would like to avoid using
infinitely many of the new variables $Y_\al$), he restricted the degree of the polynomials of the
added redundant inequalities.

\bigskip\noindent
Therefore fix a degree bound $d\in\N_0$ and set
$g_0:=1\in\R\subseteq\R[\x]$. For each $i\in\{0,\dots,m\}$ with $g_i\ne0$, fix a (column)
vector $v_i$ whose entries are
the different
monomials of degree at most
\begin{align}\label{rieq}
r_i:=\frac{d-\deg g_i}2
\end{align}
and set $\ell_i:=\dim\R[\x]_{r_i}$.
Note that in the case $g_i\notin\R[\x]_d$, $r_i$ is negative,
and consequently $\ell_i=0$ and $v_i=()\in\R[\x]^0=\{0\}$ is the empty vector.
This case is usually avoided in practice and in the literature
by assuming $d$ large enough but we think it is
more convenient to admit it. 
In the pathological case $g_i=0$, we set $r_i:=-\infty$, $\ell_i:=0$ and
let $v_i$ again be the empty vector.
Then
\[\R[\x]_{r_i}=\{a^Tv_i\mid a\in\R^{\ell_i}\}\]
and
\[\{p^2g_i\mid p\in\R[\x]_{r_i}\}=\{(a^Tv_i)^2g_i\mid a\in\R^{\ell_i}\}=\{a^T(g_iv_iv_i^T)a\mid a\in\R^{\ell_i}\}.\]
The key observation is that instead of linearizing each $p^2g_i$ with $p\in\R[\x]_{r_i}$
individually, we can just linearize the symmetric matrix polynomial
$g_iv_iv_i^T\in\R[\x]^{\ell_i\times\ell_i}$. In this way, we get for each $i\in\{0,\dots,m\}$ a
linear symmetric matrix polynomial $M_i\in\R[\x,(Y_\al)_{2\le|\al|\le d}]_1^{\ell_i\times\ell_i}$.
Instead of an \emph{infinite} family of linear inequalities, we thus get \emph{finitely many}
linear \emph{matrix} inequalities \cite{befb} (whose size depends on $d$) saying that 
\[M_0(x,y)\succeq 0,\dots,M_m(x,y)\succeq0\qquad(x\in\R^n,y\in\R^I)\] 
where $I:=\{\al\in\N_0^n\mid2\le|\al|\le d\}$ and ``$\succeq0$'' means positive semidefiniteness.
By defining
$M\in\R[\x,(Y_\al)_{2\le|\al|\le d}]_1^{\ell\times\ell}$ with $\ell:=\ell_0+\dots+\ell_m$
as the block diagonal matrix with blocks $M_0,\dots,M_n$,
we could even combine this into a \emph{single} linear matrix inequality
\[M(x,y)\succeq 0\qquad(x\in\R^n,y\in\R^I).\] 
Its solution set is a spectrahedron \cite{vin} (in particular a semialgebraic
closed convex subset of $\R^n$)
that projects down to the convex set
\[(*)\qquad S_d(\g):=\{x\in\R^n\mid\exists y\in\R^I:M(x,y)\succeq0\}.\]
The description $(*)$ of $S_d(\g)$ is called the \emph{degree $d$ Lasserre relaxation} of $\g$ (or of
the system of polynomial inequalities given by $\g$).
By abuse of language, we call sometimes $S_d(\g)$ itself the degree $d$
Lasserre relaxation of $\g$. By construction, it is clear that each $S_d(\g)$ is convex and
\[S(\g)\subseteq\ldots\subseteq S_{d+2}(\g)\subseteq S_{d+1}(\g)\subseteq S_d(\g).\]
If $S(\g)$ happens to be convex, there is a certain hope that $S_k(\g)$ equals $S(\g)$ for
all $k$ large enough. In this case, we say that $\g$ (or the system of polynomial
inequalities given by $\g$) \emph{has an exact Lasserre relaxation}.

\bigskip\noindent
In this article, we provide
a new sufficient criterium for $\g$ to have an exact Lasserre relaxation. To the best of our knowledge
this is the strongest result currently available for convex $S(\g)$.

\bigskip\noindent
If $S(\g)$ is not convex, one
can still ask whether $S_k(\g)$ equals eventually \emph{the convex hull} of $S(\g)$. This seems to
require very different techniques and will be studied in our forthcoming work \cite{ks}, see also
Example \ref{twodisks} below.

\bigskip\noindent
Here we will also not address the important question asking from what $k$ on
$S(\g)$ equals $S_k(\g)$ in case $\g$ has an exact Lasserre relaxation. In principle, a corresponding
complexity analysis of our proof would probably be possible but would, at least for general $\g$,
be extremely
tedious, and in the end yield a bound that is only of theoretical interest.

\bigskip\noindent
The Lasserre relaxation $(*)$ is a special case of the more general \emph{semidefinite representation}
of a subset $S\subseteq\R^n$
\[(**)\qquad S=\{x\in\R^n\mid\exists y\in\R^{h}:H(x,y)\succeq0\}\]
where $H\in\R[\x,Y_1,\dots,Y_h]_1^{\ell\times\ell}$ is a symmetric linear matrix polynomial
for some $h,\ell\in\N_0$. Sets $S$ having such a representation $(**)$ are called \emph{semidefinitely
representable}. Other commonly used terms are
\emph{projections of spectrahedra}, \emph{spectrahedral shadows}, \emph{spectrahedrops}, \emph{lifted LMI sets} and
\emph{SDP-representable sets}.
If the number $h$ of additional variables
is not too large, one can optimize efficiently linear functions on such sets
by the use of semidefinite programming, an important
generalization of linear programming \cite{nn}.
Semidefinitely representable sets are obviously convex and they are semialgebraic by Tarski's
real quantifier elimination \cite[Thm. 2.1.6]{pd}. The class of semidefinitely representable sets
is closed under many operations like for example taking the interior \cite{net}.
It was asked by Nemirovski
in his plenary address at the 2006 International Congress of Mathematicians in Madrid
whether each convex semialgebraic set is semidefinitely representable \cite[Subsection 4.3.1]{nem}.
Helton and Nie conjectured the answer to be positive \cite[Section 6]{hn2}. In two seminal works,
Scheiderer proved this conjecture for $n=2$ \cite[Theorem 6.8]{s1} and very recently disproved it
for each $n\ge 14$ \cite[Remark 4.21]{s2}.

\bigskip\noindent
In \cite[Theorem 3.5]{nps}, it has been shown that $\g$ cannot have an exact Lasserre relaxation if
$S(\g)\subseteq\R^n$ is convex, has nonempty interior and has
at least one non-exposed face. Other obstructions
to exactness have been given by Gouveia and Netzer \cite{gn}, see Theorem \ref{gnobstr} below.

\bigskip\noindent
On the positive side, the breakthrough was the seminal work of Helton and Nie \cite{hn2} from 2009
preceded by their earlier work \cite{hn1}, which curiously appeared later. We will the summarize the strategy behind their approach, which builds on ideas of Lasserre \cite{l2}, and indicate where this paper introduces advantageous modifications:

\bigskip\noindent
Let $\g:=(g_1,\dots,g_m)\in\R[\x]^m$
and suppose $S(\g)$ is convex and has nonempty interior.
We will introduce in Definition \ref{dftrunc} below the 
\emph{$d$-truncated quadratic module} $M_d(\g)$ associated to $\g$.
It consist of \emph{the sums} of polynomials $p^2g_i$ with $\deg(p^2g_i)\le d$ (or equivalently $\deg(p)\le r_i$, see Equation \eqref{rieq} above). 
As explained above, these were the polynomials that we add before the linearization when we build the degree $d$ Lasserre relaxation.
The following fact is good to know although
we will need from it only the trivial ``if'' part in order to prove our Main Theorem \ref{mainthm}:
We have $S(\g)=S_d(\g)$
if and only if all $f\in\R[\x]_1$ (i.e., all linear polynomials) that are nonnegative on $S(\g)$ lie
in $M_d(\g)$, see Proposition \ref{duality} below.

\bigskip\noindent
Denoting by $M(\g)=\bigcup_{d\in\N}M_d(\g)$ the quadratic module generated by $\g$ introduced in Definition \ref{dftrunc} below,
one deduces from this (due to the compactness of $S$)
a trivial necessary
condition for $\g$ having an exact Lasserre relaxation: For each $f\in\R[\x]_1$, there is an $N\in\N$
such that $f+N\in M(\g)$. If $\g$ satisfies this condition, one says that $M(\g)$ is Archimedean,
see Proposition \ref{chararch}(d) below.
This condition is unfortunately stronger than compactness of $S(\g)$. In
practice, this is however not too important, since a small change of the description $\g$ of $S(\g)$
always makes $M(\g)$ Archimedean if $S(\g)$ is compact, see Remark \ref{notmuchstronger} below.

\bigskip\noindent
Therefore suppose for the rest of the introduction that $M(\g)$ is Archimedean.

\bigskip\noindent
We saw that it suffices to look at those
$f\in\R[\x]_1$ nonnegative on $S(\g)$ whose real zero set is a supporting hyperplane of the convex set $S(\g)$. By Putinar's Positivstellensatz from 1993 (see \cite[Lemma 4.1]{put}, \cite[Thm. 5.3.8]{pd},
\cite[Cor. 5.6.1]{mar}, \cite{lau}), we know that each
$f\in\R[\x]$ \emph{positive} on $S(\g)$ lies in $M(\g)$. However, this is not really what we need here.
The advantage we have is that we need to consider only $f\in\R[\x]_1$, i.e., only linear polynomials.
The problem we have to fight is however that we have $f$ only \emph{nonnegative} on $S(\g)$ and,
most importantly, we need a uniform degree bound $d$ for which all such $f$ are in one and the same
$M_d(\g)$. Such degree bounds are known for polynomials positive on $S(\g)$ but depend on
a measure of how close $f$ comes to have a zero on $S(\g)$ \cite[Theorem 6]{ns'}.

\bigskip\noindent
Lasserre \cite{l2} made a first key observation to deal with this problem: He considered without loss
of generality only such $f\in\R[\x]_1$ nonnegative on $S(\g)$ that vanish in at least one point $u\in S(\g)$
(and whose real zero set therefore defines a supporting hyperplane at the point $u$
of the convex set $S(\g)$ unless $f=0$). Under a very restrictive condition, namely that
the Hessians of the defining
polynomials $g_i$ have a certain matrix sums-of-squares (\emph{sos} for short) representation
(and in particular, are globally concave, which is still very restrictive), he showed that he can
produce from this finitely many matrix sos representations by the use of Karush–Kuhn–Tucker
(KKT) multipliers (the Lagrange multiplier technique for inequalities instead of equations
\cite[Section 2.2]{fh}).

\bigskip\noindent
In the aforementioned articles \cite{hn1,hn2}, Helton and Nie pushed the idea of Lasserre much further and made it fruitful in many situations.
There are several important ideas in their work.
For those Hessians of the $g_i$ for which the matrix sos certificate that Lasserre assumed
(and which is trivial for those $g_i$ that happen to be \emph{linear}) 
does not exist, they show that in many situations, one can with a lot of new ideas still
pursue the basic strategy of Lasserre. These ideas include:
\begin{itemize}
\item One might exchange in a very subtle way the $g_i$ at certain places by suitable
$h_i$ having stronger concavity properties.
\item Instead of looking for matrix sos representations of the Hessians
themselves, they look for matrix representations of certain matrix polynomials arising from
double integrals of the Hessians and depending on a parameter $u$ that runs over part of the boundary of $S(\g)$. The matrix polynomial
belonging to this parameter $u$ serves to produce the bounded degree
polynomial sos certificates for those linear polynomials $f$ defining
a supporting hyperplane containing the point $u$.
\item Instead of \emph{assuming} the sos certificates as Lasserre did, Helton and Nie
had the idea to \emph{prove} the existence using a matrix version of Putinar's Positivstellensatz
that was already available \cite[Thm. 2]{sh}. Because of the dependence of the tangent point $u$
of the supporting hyperplane, they had to prove a version of Putinar's theorem for matrix
polynomials with degree bounds similar to the one existing already for polynomials
that was mentioned above (see \cite[Thm. 29]{hn1} and Theorem \ref{putinarmatrixbound} below).
\end{itemize}
We modify the approach of Helton and Nie at several places, but the most important change
is a new analysis of the properties of the modified polynomials $h_i$ which are at the same time chosen
slightly more carefully (see Lemma \ref{guess} below). This new analysis shows that the
double integral mentioned above (actually already a related single integral) is negative definite
even if the term under the integral is not negative semidefinite on the whole domain of integration,
see Lemma \ref{pain} below.
Helton and Nie  seem to be compelled to work with negative semidefinite terms under the integral
whereas the new method enables us to be more liberal about this issue.

\bigskip\noindent
In this way, we will be able to show our Main Theorem \ref{mainthm}:
If each $g_i$ satisfies a certain 
second order strict quasiconcavity condition (see Definition \ref{dfqc} below)
where it vanishes on
$S(\g)$ (which is very natural because of the convexity of $S$, see Proposition \ref{qcc}(b) below) 
or its Hessian has a matrix sos certificate
for negative-semidefiniteness on $S$ (see Definition \ref{dftrunc} below),
then $\g$ has an exact Lasserre relaxation.

\bigskip\noindent
Helton and Nie showed under the same conditions only that $S(\g)$ is semidefinitely representable
\cite[Thm. 3.3]{hn2}.
They obtained the semidefinite representation by glueing together
Lasserre relaxations of many small pieces obtained in a non-constructive way
\cite[Prop. 4.3]{hn2} (see also \cite{ns}).
With a very tedious proof (using smoothening techniques similar to those from \cite{gho})
they show in addition under very technical assumptions not easy to state \cite[Section 5]{hn2}
that there exists $s\in\N_0$ and $\h\in\R[\x]^s$ such that $S(\g)=S(\h)$ and $\h$ has an exact
Lasserre relaxation \cite[Theorem 5.1]{hn2}. In his diploma thesis, Sinn thoroughly analyzed and improved this proof and showed under the same technical assumptions that one can take
$\h:=(g_1,\dots,g_m,g_1g_2,g_1g_3,\dots,g_{m-1}g_m)$ \cite[Theorem 3.3.2]{sin}.


\section{Reminder on sums of squares}

\noindent
In this section, we collect all the tools from the interplay between positive polynomials and sums of
squares that we need from the area of real algebraic geometry.

\begin{df}\label{sosdef}
We call $p\in\R[\x]$ a \emph{sums-of-squares (sos)} polynomial
if there exist $\ell\in\N_0$ and polynomials $p_1,\dots,p_\ell\in\R[\x]$
such that
\[p=p_1^2+\ldots+p_\ell^2.\]
\end{df}

\noindent
We say that a polynomial $p\in\R[\x]$ is \emph{nonnegative} (or \emph{positive}) on a set
$S\subseteq\R^n$ if $p(x)\ge0$ (or $p(x)>0$) for all $x\in S$. In this case, we write
``$p\ge0$ on $S$'' (or ``$p>0$ on $S$'').

\bigskip\noindent
It is obvious that each sos polynomial is nonnegative on $\R$.
In Lemma \ref{guess} below, we will need the well-known fact that each polynomial
in \emph{one} variable nonnegative on $\R$ is sos.

\begin{pro}\label{sos}
Let $f\in\R[T]$ with $f\ge0$ on $\R$. Then $f$ is sos.
\end{pro}

\begin{proof} Using the fundamental theorem of algebra, one shows easily that there
are $p,q\in\R[T]$ such that $f=(p-\ii q)(p+\ii q)=p^2+q^2$ where
$\ii:=\sqrt{-1}\in\C$ is the imaginary unit.
\end{proof}

\noindent
A matrix $A\in\R^{k\times k}$ is called \emph{positive semidefinite (psd)} (or
\emph{positive definite (pd)}) if it is symmetric and
$x^TAx\ge0$ (or $x^TAx>0$) for all $x\in\R^k \setminus \{0\}$.
Equivalently, $A$ is symmetric and the eigenvalues of $A$ (which are all real) are all
nonnegative (or positive). In this case, we write $A\succeq0$ (or $A\succ0$). By $A\succeq B$,
$A\succ B$, $A\preceq 0$ etc., we mean $A-B\succeq0$, $A-B\succ0$, $-A\succeq0$ and so on.

\bigskip\noindent
The appropriate generalization of Definition \ref{sosdef} to matrix polynomials is the following.

\begin{df}
We call $P\in\R[\x]^{k\times k}$ a \emph{sums-of-squares (sos)} matrix polynomial
if there exist $\ell\in\N_0$ and $P_1,\dots,P_m\in\R[\x]^{k\times k}$
such that
\[P=P_1^TP_1+\ldots+P_\ell^TP_\ell.\]
\end{df}

\bigskip\noindent
The following is an easy exercise that is good to know when dealing with sos matrix polynomials.

\begin{pro} For $P\in\R[\x]^{k\times k}$, the following are equivalent:
\begin{enumerate}[(a)]
\item $P$ is an sos matrix.
\item There is an $\ell\in\N_0$ and a matrix polynomial $Q\in\R[\x]^{\ell\times k}$ such
that $P=Q^TQ$.
\item There are $\ell\in\N_0$ and $v_1,\dots,v_\ell\in\R[\x]^k$ such that
$P=v_1v_1^T+\ldots+v_\ell v_\ell^T$.
\end{enumerate}
\end{pro}

\noindent
We say that a matrix polynomial $P\in\R[\x]^{k\times k}$
is \emph{psd} (or \emph{pd}) on a set
$S\subseteq\R^n$ if $P(x)\succeq0$ (or $P(x)\succ0$) for all $x\in S$. In this case, we write
``$P\succeq0$ on $S$'' (or ``$P\succ0$ on $S$'').

\begin{df}\label{dfqm}
A subset $M$ of $\R[\x]$ is called a \emph{quadratic module} of $\R[\x]$ if
\begin{itemize}
\item $1\in M$,
\item $p+q\in M$ for all $p,q\in M$ and
\item $p^2q\in M$ for all $p\in\R[\x]$ and $q\in M$.
\end{itemize}
For a tuple $\g:=(g_1,\dots,g_m)\in\R[\x]^m$, the smallest quadratic module containing $g_1,\dots,g_m$
is obviously
\[M(\g):=\left\{\sum_{i=0}^m s_ig_i\mid\text{$s_0,\dots,s_m\in\R[\x]$ are sos}\right\}\]
where we set $g_0:=1$. We call it the quadratic module \emph{generated} by $\g$.
\end{df}

\begin{df}
A quadratic module $M$ of $\R[\x]$ is called \emph{Archimedean} if for all $p\in M$ there is some
$N\in\N$ such that $N+p\in M$.
\end{df}

\noindent
The following is well-known (see for example \cite[Lemma 5.1.13]{pd} and \cite[Cor. 5.2.4]{mar})
but for convenience of the reader we include a compact easy proof.

\begin{pro}\label{chararch}
Let $M$ be a quadratic module of $\R[\x]$. Then the following are equivalent:
\begin{enumerate}[(a)]
\item $M$ is Archimedean.
\item There is some $N\in\N$ such that $N-(X_1^2+\ldots+X_n^2)\in M$.
\item There are $m\in\N$ and $\g\in(\R[\x]_1\cap M)^m$
such that the polyhedron $S(\g)$ is non-empty and compact.
\item For each $f\in\R[\x]_1$, there is some $N\in\N$ such that $N+f\in M$.
\end{enumerate}
\end{pro}

\begin{proof} Consider the vector subspace
\[B:=\{p\in\R[\x]\mid\exists N\in\N:N\pm p\in M\}\supseteq\R\]
of $\R[\x]$. If $p\in\R[\x]$ with $p^2\in B$, then
we can choose $N\in\N$ such that $(N-1)-p^2\in M$ and thus
\[N\pm p=(N-1)-p^2+\left(\frac12\pm p\right)^2+\frac{3}{4}\in M\]
and thus $p\in B$. Conversely, if $p\in B$, then one can choose $N\in\N$ such that $2N-1\pm p\in M$ and thus
\[N^2(2N-1)-p^2=\frac{1}{2}\left((N-p)^2(2N-1+p)+(N+p)^2(2N-1-p)\right)\in M,\]
showing that $p^2\in B$ since anyway $N^2(2N-1)+p^2\in M$. Thus, we have
\[(*)\qquad p^2\in B\iff p\in B\]
for all $p\in\R[\x]$. This implies that $B$ is a subring of $\R[\x]$. Indeed, for $p,q \in \R[\x]$ with $p,q \in B$ we have 
\[pq=\frac12(\overbrace{(\underbrace{p+q}_{\in B})^2}^{\in B}-\underbrace{p^2}_{\in B}-\underbrace{q^2}_{\in B})\in B.\]
This shows that $\R[\x]_1\subseteq B\iff\R[\x]=B$, which is the equivalence
(d)$\iff$(a). Condition (b) is easily seen to be equivalent to $X_1^2,\dots,X_n^2\in B$, which in turn
is by $(*)$ equivalent to $X_1,\dots,X_n\in B$. Again by using that $B$ is a subring of $\R[\x]$, this shows
the equivalence (a)$\iff$(b). It remains to show (c)$\iff$(d). If (d) holds, then one trivially finds $\g$ like
in (c), e.g., with $S(\g)$ being a hypercube. Conversely, suppose that we have $\g$ like in (c) and let
$f\in\R[\x]_1$. Then there is $N\in\N$ such that $N+f\ge0$ on the polytope $S(\g)$.
By the affine form of Farkas' lemma \cite[Cor. 7.1h, p. 93]{scr}, we have that
$N+f$ is a nonnegative linear combination of the $1,g_1,\dots,g_m$ and thus lies in $M$.
\end{proof}

\noindent
We mention the following important theorem although we will need it only for Example \ref{twodisks}
below.

\begin{thm}[Schmüdgen]\label{schmu}
Let $M$ be a quadratic module of $\R[\x]$. The following are equivalent:
\begin{enumerate}[(a)]
\item There are $m\in\N$ and $\g=(g_1,\dots,g_m)\in\R[\x]^m$ such that $S(\g)$ is compact and
$\prod_{i\in I}g_i\in M$ for all $I\subseteq\{1,\dots,m\}$.
\item There is some $g\in M$ with compact $S(g)$.
\item $M$ is Archimedean.
\end{enumerate}
\end{thm}

\begin{proof}(a)$\implies$(c) is the deep part of Schmüdgen's Positivstellensatz \cite[Cor. 3]{scm}, namely
his characterization of Archimedean \emph{preorders}  (see \cite[Thm. 5.1.17]{pd} and \cite[Thm. 6.1.1]{mar}).
The implications (c)$\implies$(b)$\implies$(a) are trivial.
\end{proof}

\begin{rem}\label{notmuchstronger}
For $n\ge2$, there are examples of $\g=(g_1,\dots,g_m)\in\R[\x]^m$
with compact (even empty) $S(\g)$ such that $M(\g)$ is not Archimedean
(see \cite[Ex. 7.3.1]{mar} or \cite[Ex. 6.3.1]{pd}). However if $S(\g)$ is compact, then
Proposition \ref{chararch} and Theorem~\ref{schmu} provide several ways of changing
the description $\g$ of $S(\g)$ such that $M(\g)$ becomes Archimedean.
For example, if one knows a big ball containing $S(\g)$, it suffices to add its defining
quadratic polynomial to $\g$ by Proposition \ref{chararch}(b).
That is why for many practical purposes, the Archimedean
property of $M(\g)$ is not much stronger than the compactness of $S(\g)$.
\end{rem}

\noindent
We use the symbols $\nabla$ and $\hess$ to denote the gradient and the Hessian of a real-valued function of
$n$ variables, respectively. For a \emph{polynomial} $g\in\R[\x]$, we understand its gradient $\nabla g$ as a
column vector from $\R[\x]^n$, i.e., as a vector of polynomials. Similarly, its Hessian
$\hess g$ is a symmetric matrix polynomial of size $n$, i.e., a symmetric matrix from $\R[\x]^{n\times n}$.

\begin{df}\label{dftrunc}
Let $\g:=(g_1,\dots,g_m)\in\R[\x]^m$ and set again $g_0:=1$. For $i\in\{0,\dots,m\}$, set
$r_i:=\frac{d-\deg g_i}2$ if $g_i\ne0$ and $r_i:=-\infty$ if $g_i=0$.
Then we define the \emph{$d$-truncated quadratic module} $M_d(\g)$
associated to $\g$ by
\[
M_d(\g):=
\left\{\sum_{i=0}^m\sum_jp_{ij}^2g_i\mid p_{ij}\in\R[\x]_{r_i}\right\}\subseteq M(\g)\cap\R[\x]_d.
\]
More generally, we define the \emph{$d$-truncated $k\times k$ matricial quadratic module}
associated to $\g$ by
\[
M_d^{k\times k}(\g):=
\left\{\sum_{i=0}^m\sum_jP_{ij}^TP_{ij}g_i\mid P_{ij}\in\R[\x]_{r_i}^{k\times k}\right\}\subseteq
\R[\x]^{k\times k}_d.
\]
We say that $f\in\R[\x]$ is \emph{$\g$-sos-concave} if
\[-\hess f\in M^{n\times n}(\g):=\bigcup_{d\in\N_0}M_d^{n\times n}(\g).\]
If $m=0$, this means that the negated Hessian of $f$ is an sos matrix polynomial and we say that $f$ is
\emph{sos-concave}.
\end{df}

\noindent
Any $f\in\R[\x]_1$ is sos-concave since $\hess f=0$. The Hessian of a $\g$-sos-concave polynomial
is negative semidefinite on $S(\g)$. 

\bigskip\noindent
The following is Putinar's Positivstellensatz \cite[Lemma 4.1]{put} for matrix polynomials with
degree bounds. It has been first proven by Helton and Nie \cite[Thm. 29]{hn1} following the
technical approach of
Nie and the second author \cite{ns'} for the case of polynomials. This technical approach yields
explicit degree bounds. The first author found a short topological proof for the mere existence of
such bounds \cite[Thm. 3.2]{kri} that is based on knowing already the result without the
degree bounds that stems from \cite[Thm. 2]{sh}.

\begin{thm}[Helton and Nie]\label{putinarmatrixbound}
Fix $C,d,k,m,n \in \N$ and fix any norm on the vector space $\R[\x]_d^{k \times k}$.
Let $\g:=(g_1,\dots,g_m)\in\R[\x]^m$ such that $M(\g)$ is Archimedean.
Then there exists $d\in\N_0$ such that every symmetric $H\in\R[\x]_d^{k \times k}$
satisfying $\|H\|\le C$ and $H\succeq\frac1C$ on $S(\g)$ satisfies $H\in M_d^{k\times k}(\g)$.
\end{thm}

\noindent
The following is a slight generalization of \cite[Lemma 7]{hn1} that will be needed in the proof of
Theorem \ref{linearstability}.

\begin{lem}\label{bochnercone}
Let $d\in\N_0$, $\g:=(g_1,\dots,g_m)\in\R[\x]^m$ and $u\in\R^n$. If $P\in M_d^{k\times k}(\g)$, then
the matrix polynomial $H\in\R[\x]^{k\times k}$ defined by
\[H(x)=\int_0^1\int_0^tP(u+s(x-u)) \ ds\ dt\]
for $x\in\R^n$ lies again in $M_d^{k\times k}(\g)$.
\end{lem}

\begin{proof}
The proof \cite[Lemma 7]{hn1} can be easily adapted. Another more conceptual proof is the following:
$M_d^{k\times k}(\g)$ is a convex cone in a finite-dimensional vector space. Then
\[H=\int_0^1\int_0^tP(u+s(\x-u)) \ ds\ dt\]
is an existing Bochner integral of a vector valued function with values in this convex
cone and thus lies again in this convex cone \cite{rw} (regardless of whether the cone is closed or not).
\end{proof}

\noindent
The ``if'' direction of the following proposition is trivial since a closed convex set in a finite-dimensional vector
space is the intersection over all half spaces containing it. We will use it to prove our Main Theorem
\ref{mainthm}. The ``only if'' direction will be needed only in Example \ref{twodisks} below.

\begin{pro}[Netzer, Plaumann and Schweighofer] \label{duality}
Suppose $d\in\N_0$, $\g:=(g_1,\dots,g_m)\in\R[\x]_d^m$, $S(\g)$ is compact and convex and
has nonempty interior.
Then $S_d(\g)=S(\g)$ if and only if every $f \in \R[\x]_1$ with $f\ge0$ on $S(\g)$ lies
in $M_d(\g)$.
\end{pro}

\begin{proof}
This is a special case of  \cite[Proposition 3.1]{nps}.
\end{proof}


\section{Reminder on strict quasiconcavity}

\noindent
We denote the \emph{real zero set} of $g$ by
\[Z(g):=\{x\in\R^n\mid g(x)=0\}.\]
We adopt the following notion from \cite[p. 25]{hn1}, which is a
local second order quasiconcavity condition.

\begin{df}\label{dfqc}
Let $g\in\R[\x]$.
We say that $g$ is \emph{strictly quasiconcave} at $x\in\R^n$ if
for all $v\in\R^n\setminus\{0\}$ with $(\nabla g(x))^Tv=0$, we have that
$v^T(\hess g(x))v<0$. We say that $g$ is \emph{strictly quasiconcave} on $A\subseteq\R^n$
if $g$ is \emph{strictly quasiconcave} at each point of $A$.
\end{df}

\begin{rem}
Let $g\in\R[\x]$ and $x\in\R^n$ such that $\nabla g(x)=0$.
\begin{enumerate}[(a)]
\item
$g$ is strictly quasiconcave at $x$ if and only if
$\hess g(x)\prec0$.
\item
If $g$ is strictly quasiconcave at $x$ and $g(x)=0$, then there is a neighborhood $U$ of $x$ such that
$U\cap S(g)=\{x\}$.
\end{enumerate}
\end{rem}

\noindent
If $g\in\R[\x]$ satisfies $g(x)=0$ and $\nabla g(x)\ne0$, then $Z(g)$ is locally around $x$ a
smooth hypersurface. Differential geometers will recognize that strict quasiconcavity of $g$
at $x$ then means that the second fundamental form of this hypersurface at $x$ is positive
definite when one chooses the ``outward normal'' (pointing away from $S(g)$). Thus this
means that $S(g)$ is locally convex in a strong second order sense. For a detailed discussion we refer
to \cite{hn1,hn2} and the references therein. As Helton and Nie in
\cite[Subsection 3.1]{hn1}, we want however to help those readers who are not familiar with
the basics of differential geometry by discussing strict quasiconcavity in an elementary
manner. The reason why we include this is that Helton and Nie presuppose already that the
reader is familiar with the geometric notion of tangent hyperplanes and knows that the gradient
is a normal vector for it \cite[p. 786]{hn2}. Conversely we fit this into their arguments,
see Part (a) of the following lemma and Proposition \ref{qcc}(b) below.

\bigskip\noindent
Formally, we will
use the following lemma and the next proposition only in Example \ref{twodisks} below and
even there it can be avoided by some calculations. Some readers might therefore decide to skip
them.

\begin{lem}\label{prepqcc}
Let $n\in\N$, $g\in\R[\x]$ and $x\in\R^n$ such that $g(x)=0$ and $\nabla g(x)\ne0$.
Suppose $v_1,\dots,v_n$ form a basis of $\R^n$, $U$ is an open neighborhood of $0$ in
$\R^{n-1}$, $\ph\colon U\to\R$ is smooth and satisfies $\ph(0)=0$ as well as
\[(*)\qquad g(x+\xi_1v_1+\ldots+\xi_{n-1}v_{n-1}+\ph(\xi)v_n)=0\]
for all $\xi=(\xi_1,\dots,\xi_{n-1})\in U$. Then the following hold:
\begin{enumerate}[(a)]
\item $(\nabla g(x))^Tv_1=\ldots=(\nabla g(x))^Tv_{n-1}=0\iff
\nabla\ph(0)=0$
\item If $\nabla\ph(0)=0$ and $(\nabla g(x))^Tv_n>0$, then
\[\text{$g$ is strictly quasiconcave at $x$}\iff\hess\ph(0)\succ0.
\]
\end{enumerate}
\end{lem}

\begin{proof}
Taking the derivative of $(*)$ with respect to $\xi_i$, we get
\[(**)\qquad
(\nabla g(x+\xi_1v_1+\ldots+\xi_{n-1}v_{n-1}+\ph(\xi)v_n))^T\left(v_i+\frac{\partial\ph(\xi)}{\partial\xi_i}v_n\right)=0\]
for all $i\in\{1,\dots,n-1\}$. Setting here $\xi$ to $0$,
we get \[(\nabla g(x))^T\left(v_i+\frac{\partial\ph(\xi)}{\partial\xi_i}\middle|_{\xi=0}v_n\right)=0\] 
for each $i\in\{1,\dots,n-1\}$. From this, (a) follows easily (for ``${\implies}$'' use that $(\nabla g(x))^Tv_n\ne0$ since $v_1,\dots,v_n$ is a basis).
Taking the derivative of $(**)$ with respect to $\xi_j$, we get
\begin{multline*}
\left(v_j+\frac{\partial\ph(\xi)}{\partial\xi_j}v_n\right)^T(\hess g(x+\xi_1v_1+\ldots+\xi_{n-1}v_{n-1}+\ph(\xi)v_n))\left(v_i+\frac{\partial\ph(\xi)}{\partial\xi_i}v_n\right)\\
+(\nabla g(x+\xi_1v_1+\ldots+\xi_{n-1}v_{n-1}+\ph(\xi)v_n))^T
\left(\frac{\partial^2\ph(\xi)}{\partial\xi_i\partial\xi_j}v_n\right)=0
\end{multline*}
for all $i,j\in\{1,\dots,n-1\}$.
To prove (b), suppose now that $\nabla\ph(0)=0$ and $(\nabla g(x))^Tv_n>0$. Then the preceding equation implies
\[\hess\ph(0)=-\frac1{(\nabla g(x))^Tv_n}(v_i^T(\hess g(x))v_j)_{i,j\in\{1,\dots,n-1\}}.\]
Since $v_1,\ldots,v_{n-1}$ now form a basis of the orthogonal complement of $\nabla g(x)$ by (a), the matrix
$(v_i^T(\hess g(x))v_j)_{i,j\in\{1,\dots,n-1\}}$ is negative definite if and only if $g$ is strictly quasiconcave at $x$
(see Definition \ref{dfqc}).
\end{proof}

\noindent
The following proposition is important for understanding the notion of quasiconcavity. It is trivial that
quasiconcavity of a polynomial $g$ at $x$ depends only on the function $V\to\R,\ x\mapsto g(x)$
where $V$ is an arbitrarily small neighborhood of $x$. But if $g(x)=0$ and
$\nabla g(x)\ne0$, then it actually
depends only on the function
\[V\to\{-1,0,1\},\ x\mapsto\sgn(g(x))\]
as the equivalence of Conditions (a) and (b) of the following proposition show.

\begin{pro}\label{qcc}
Let $n\in\N$, $g\in\R[\x]$ and $x\in\R^n$ such that
\[g(x)=0\text{ and }\nabla g(x)\ne0.\]
Suppose that $V$ is a neighborhood of $x$. Then the following are equivalent:
\begin{enumerate}[(a)]
\item $g$ is strictly quasiconcave at $x$.
\item There is a basis $v_1,\dots,v_n$ of $\R^n$, an open neighborhood $U$ of $0$ in $\R^{n-1}$
and a smooth function $\ph\colon U\to\R$ such that $\ph(0)=0$, $\nabla\ph(0)=0$,
$\hess\ph(0)\succ0$,
\[(*)\qquad x+\xi_1v_1+\ldots+\xi_{n-1}v_{n-1}+\ph(\xi)v_n\in Z(g)\cap V\]
for all $\xi\in U$ and
\[(**)\qquad x+\la v_n\in S(g)\cap V\]
for all small enough $\la\in\R_{>0}$.
\item Condition (b) holds with ``basis'' replaced by ``orthogonal basis''.
\end{enumerate}
For any basis $v_1,\dots,v_n$ of $\R^n$ like in (b), one has
\[(***)\qquad(\nabla g(x))^Tv_1=\ldots=(\nabla g(x))^Tv_{n-1}=0\qquad\text{and}\qquad(\nabla g(x))^Tv_n>0.
\]
\end{pro}

\begin{proof}
Using Lemma \ref{prepqcc}(a), it is easy to show that any $v_1,\dots,v_n$ like in (b) satisfy $(***)$ using that $(\nabla g(x))^Tv_n=0$ would
contradict the hypothesis $\nabla g(x)\ne0$ since $v_1,\dots,v_n$ is a basis. Now Part (b) of the same lemma shows that (b) implies (a).
Since it is trivial that (c) implies (b), it only remains to show that (a) implies (c).

To this end,
let (a) be satisfied. In order to show (c), choose an orthogonal basis $v_1,\dots,v_n$ of $\R^n$ satisfying $(***)$.
The implicit function theorem yields an open neighborhood
$U$ of the origin in $\R^{n-1}$ such that for each $\xi=(\xi_1,\dots,\xi_{n-1})\in U$ there is a
unique $\ph(\xi)\in\R$ satisfying $(*)$, in particular $\ph(0)=0$. Moreover, one can choose $U$ such that the resulting function
$\ph\colon U\to\R$ is smooth. From $(\nabla g(x))^Tv_n>0$, we get $(**)$. From Part (a) of Lemma \ref{prepqcc}, we get
$\nabla\ph(0)=0$. From Part (b) of the same lemma and from (a), we obtain $\hess\ph(0)\succ0$.
\end{proof}

\noindent
Another more algebraic way of understanding strict quasiconcavity is given by
the following easy exercise \cite[Lemma 11(a)]{hn1}.

\begin{lem}\label{mdfhessian}
Let $S\subseteq \R^n$ be a compact set and consider a polynomial $g \in \R[\x]$ that is strictly quasiconcave on $S$. Then one can find $\la>0$ such that
\[\la \nabla g(\nabla g)^T-\hess g\] is positive definite on $S$.
\end{lem}

\noindent
We will need the following lemma only in the case where $f$ is linear. In that case, one can use for
its proof a slightly weaker version of the Karush-Kuhn-Tucker theorem \cite[Theorem 5.1]{pla}.

\begin{lem}\label{kkt}
Suppose $\g:=(g_1,\dots,g_m)\in\R[\x]^m$, $S(\g)$ is convex and has nonempty interior.
Suppose $u \in S(\g)$ and let $I:=\{i \in\{1,\dots,m\} \mid g_i(u)=0\}$.
Suppose $f\in\R[\x]$ and $U$ is a neighborhood of $u$ such that
$u$ is a minimizer of $f$ on $S(\g)\cap U$ and
$\hess g_i\preceq0$ on $S(\g)\cap U$ for all $i\in I$.
Then there exist a family $(\la_i)_{i \in I}$ of nonnegative Lagrange multipliers $\la_i\in\R_{\ge0}$ such that
$\nabla f(u)=\sum_{i \in I}\la_i\nabla g_i(u)$. 
\end{lem}

\begin{proof}
By the Karush-Kuhn-Tucker theorem \cite[Theorem 2.2.5]{fh}, it suffices to show that the $g_i$ ($i \in I$)
satisfy the Mangasarian-Fromowitz constraint qualification, i.e., there is some $v\in\R^n$ such
that $(\nabla g_i(u))^Tv>0$ for all $i\in I$ \cite[Chapter 2.2.5]{fh}. By discarding those
$g_i$ that are the zero polynomial, we may assume $g_i\ne0$ for all $i\in I$.
Since $S(\g)$ has
nonempty interior, there is then some $x\in S(\g)$ such that $g_i(x)>0$ for all $i\in I$.
Set $v:=x-u$ and consider for fixed $i\in I$
the function $h\colon\R\to\R,\ t\mapsto g_i(u+tv)$. We have $0=h(0)$ and $h(1)=g_i(x)>0$. Therefore
there is $t\in[0,1]$ such that $h'(t)>0$. Because of $h''(t)=v^T(\hess g_i(u+tv))v\le0$ for all $t\in[0,1]$,
this implies $(\nabla g_i(u))^Tv=h'(0)>0$ as desired.
\end{proof}


\section{The main result}

\noindent
In this section, we will prove our main result about the exactness of the Lasserre relaxation. The first step is to get an alternate description of the compact
basic closed semialgebraic set
$S(\g)$ with nonempty interior. Both descriptions, the original one $\g$ and the alternate one will be used in the proof of Theorem \ref{linearstability}. 
The new description will arise by replacing polynomials $g_i$ that are strictly quasiconcave on $S(\g) \cap Z(g_i)$ by
polynomials of the form $h_i:=g_ih(g_i)$ with a univariate polynomial $h\in\R[T]$ such that $h\ge1$ on $\R$. 
It will be of outmost importance that $h_i\in M(\g)$ which follows from the fact that $h-1$ and therefore $h$ is an sos-polynomial by Lemma \ref{sos} above.
Roughly speaking, the basic idea is that $h_i(x)$ will be, up to positive factor, approximately $1-e^{-cg_i(x)}$ for a big constant $c$
when $x$ lies in $S(\g)$ or $x$ lies sufficiently close to $S(\g)$. The effect of this is that $h_i$ will be a polynomial (unfortunately of large degree) that is very close to
being a positive constant on the ``safe part'' of $S(\g)$ consisting of the points in $S(\g)$ that are in ``safe distance'' to the boundary of $S(\g)$.
On the ``safe part'' of $S(\g)$ one can hope (and it will turn out from our actual choice of $h$) that the Hessian of the $h_i$ does not vary too quickly.
This will be crucial in the proof of Lemma \ref{pain} (the interval $J_3$ appearing there corresponds to this ``safe part'').

\bigskip\noindent
In the proof of Lemma \ref{guess} below, the auxiliary polynomial $h$ will be chosen as
$h:=f_{c,d}\in\Q[T]$ for a big real constant $c$ and a large nonnegative \emph{even} integer $d$
where $f_{c,d}$ is defined in Notation \ref{fcd} below. In \cite[Lemma 13]{hn1}, Helton and Nie use exactly the same polynomial $f_{c,d}$ except
that they do not care about the parity of the degree $d$. Lemma \ref{positivity}
below is an important observation that was probably not known to Helton and Nie. If Helton and Nie had exploited this, they could have
sharpened some of their results in \cite{hn1}. However, they would not have come close to our main result Theorem \ref{mainthm} which ultimately relies on
our new refined and subtle analysis in the proofs of Lemma \ref{pain} and Theorem \ref{linearstability} that focuses on \emph{integrals of} the Hessian of the
$h_i$ instead of the Hessians themselves.

\begin{notation}\label{fcd}
For $c>0$ and $d\in\N_0$, we denote by
\[e_{c,d}:=\sum_{k=0}^d\frac{c^k}{k!}T^k\in\Q[T]\]
the $d$-th Taylor polynomial of the function
\[\R\to\R,\ t\mapsto\exp(ct)\] at the origin and we set
\[f_{c,d}:=\frac{1-e_{c,d+1}(-T)}{cT}=\sum_{k=0}^d\frac{c^{k}}{(k+1)!}(-T)^k\in\Q[T].\]
For any $p\in\R[T]$, we denote by $p'$ its (formal) derivative (with respect to $T$) and by $p''=(p')'$ its second derivative.
\end{notation}

\begin{pro}\label{calculus}
For $c>0$, we have
\begin{align*}
\tag{a}e_{c,d}'&=ce_{c,d-1}\qquad\text{for $d\in\N$,}\\
\tag{b}f_{c,d}'&=\frac{e_{c,d}(-T)-f_{c,d}}T\qquad\text{for $d\in\N_0$ and}\\
\tag{c}f_{c,d}''&=\frac{-e_{c,d}'(-T)-2f_{c,d}'}T\qquad\text{for $d\in\N$.}
\end{align*}
\end{pro}

\begin{proof}
Use the chain rule, the product rule and the quotient rule for derivation.
\end{proof}

\noindent
The following lemma has been given an easy short proof by Speyer \cite{spe}, which we reproduce here for convenience of the reader.

\begin{lem}[Speyer]\label{speyer}
For $c\in\R_{>0}$ and $d\in\N_0$, we have:
\begin{enumerate}[(a)]
\item If $d$ is even, then $e_{c,d}(t)>0$ for all $t\in\R$.
\item If $d$ is odd, then $e_{c,d}$ is strictly increasing on $\R$.
\end{enumerate}
\end{lem}

\begin{proof}
We fix $c\in\R_{>0}$ and proceed by induction on $d$. The case $d=0$ is trivial since $e_{c,0}=1>0$. Suppose the lemma is already proven
for $d-1$ instead of $d$ where $d\in\N$ is fixed.
First consider the case where $d$ is even. Then by induction hypothesis the odd degree polynomial $e_{c,d-1}$
must have exactly one real root $t_0$. By Lemma \ref{calculus}(a) the even degree polynomial $e_{c,d}$ takes therefore
its (unique) minimum in $t_0$. To prove the statement, it suffices to observe that
\[e_{c,d}(t_0)=\frac{(ct_0)^d}{d!}+e_{c,d-1}(t_0)=\frac{(ct_0)^d}{d!}+0>0.\]
In the case where $d$ is odd, the statement follows immediately from the induction hypothesis and Lemma \ref{calculus}(a).
\end{proof}

\begin{lem}\label{positivity}
Let $c\in\R_{>0}$ and suppose $d\in\N_0$ is even. Then $f_{c,d}(t)>0$ for all $t\in\R$.
\end{lem}

\begin{proof}
The leading coefficient of $f_{c,d}$ is $\frac{c^d(-1)^d}{(d+1)!}>0$. Therefore it suffices to show that $f_{c,d}$ has no real roots.
One easily checks that $f_{c,d}$ has no root at the origin. Assume we have a root $t\in\R$ different from the origin. Then
$e_{c,d+1}(-t)=1$. Observing that $e_{c,d+1}(0)=1$, it follows from Lemma \ref{speyer}(b) that $t=0$, a contradiction.
\end{proof}

\bigskip\noindent
The following lemma is an improved version of \cite[Lemma 13]{hn1}. Most importantly, we manage to get that
$h-1$ (defined in this lemma) is an sos polynomial (and in particular $h$ is positive on $\R$) instead of just positivity of
$h$ on the interval $[0,R]$.
This will come out of Lemmata \ref{positivity} and \ref{sos} together with the approach we take in the proof
that uses simply Taylor approximations of the exponential function instead of the nonconstructive approximation theory used in
\cite{hn1}. The second crucial improvement is the new property (c). A surprising improvement coming out of Lemma
\ref{speyer} is that we get in Condition (a) positivity on $\R$ instead of just the positivity on $[0,R]$ that Helton and Nie get.
At the moment however, we do not have any application for this. Finally, an insignificant improvement again
not used by us is the validity of Condition (b) on the interval $[-R,R]$ instead of the interval $[0,R]$ used by Helton and Nie.

\begin{lem}\label{guess}
Let $H,\de,\ep,R\in\R$ such that $H>0$ and $0<\de<\ep<R$.
Then there exists a univariate polynomial $h\in\R[T]$ such that
\[\text{$h-1$ is an sos polynomial}\]
satisfying the following conditions:
\begin{gather*}
\tag{a}h(t)+th'(t)>0\text{ for all }t\in\R\\
\tag{b}2h'(t) + th''(t)<-H(h(t) + th'(t))\text{ for all }t\in[-R,R]\\
\tag{c}H\max\left\{ h(t) + th'(t) \mid t \in [\ep,R] \right\}<\min\left\{ h(t) + th'(t) \mid t \in [-R,\de]\right\}
\end{gather*}
\end{lem}

\begin{proof}
By a scaling argument, we can relax the condition that $h-1$ is sos to
the condition that $h-\gamma$ is sos for some $\gamma\in\R_{>0}$.
By the Lemmata \ref{positivity} and \ref{sos}, it suffices to find
$c\in\R_{>0}$ and $d\in\N_0$ even such that (a)--(c) are satisfied for $h:=f_{c,d}\in\Q[T]$. Noting that
\[f_{c,d}+Tf_{c,d}'=e_{c,d}(-T)\qquad\text{and}\qquad 2f_{c,d}'+Tf_{c,d}''=-e_{c,d}'(-T)=-ce_{c,d-1}(-T)\]
by Proposition \ref{calculus}, this means that we are trying to find $c\in\R_{>0}$ and $d\in\N_0$ even with
\begin{gather*}
\tag{a'}e_{c,d}(-t)>0\text{ for all }t\in\R\\
\tag{b'}-ce_{c,d-1}(-t)<-He_{c,d}(-t)\text{ for all }t\in[-R,R]\\
\tag{c'}H\max\left\{e_{c,d}(-t)\mid t\in[\ep,R] \right\}<\min\left\{e_{c,d}(-t)\mid t \in[-R,\de]\right\}.
\end{gather*}
Condition (a') is always satisfied by Lemma \ref{speyer}(a) if $d$ is even.
Since the functions induced by the polynomials $e_{c,d}$ on the interval
$[-R,R]$ converge uniformly to the function $[-R,R]\to\R,\ t\mapsto\exp(ct)$ as $d\in\N$ tends to infinity, it suffices to find $c>0$ satisfying
\begin{gather*}
\tag{b''}-c\exp(-ct)<-H\exp(-ct)\text{ for all }t\in[-R,R]\\
\tag{c''}H\max\left\{\exp(-ct) \mid t \in [\ep,R] \right\}<\min\left\{ \exp(-ct) \mid t \in [-R,\de]\right\}.
\end{gather*}
These conditions can be rewritten as
\begin{gather*}
\tag{b''}-c<-H\\
\tag{c''}H\exp(-c\ep)<\exp(-c\de).
\end{gather*}
Thus it suffices to choose $c>\max\left\{H,\frac{\log H}{\ep-\de}\right\}$ and $d \in \N_0$ even and sufficiently large.
\end{proof}

\noindent
The previous result is now used to prove the following key lemma.
This key lemma is our ``luxury version'' of \cite[Proposition 10]{hn1} in the work of Helton and Nie.
It will be used in this article only with $C:=S(\g)$ (when $S(\g)$ is compact) but for potential future applications
we formulate it in greater generality. It has several advantages over \cite[Proposition 10]{hn1}. The most
important one is that we only require the $g_i$ to be strictly quasiconcave
on a set that will be very slim in general whereas Helton and Nie assume
them to be strictly quasiconcave on the
whole of $S(\g)$. Another important advantage is that the new polynomials $h_i$ lie in $M(\g)$. The only price that we have
to pay is that not the Hessian itself but only an integrated version of it satisfies the negative definiteness
condition. This will however be enough for the proof of Theorem \ref{linearstability} and the Main Theorem
\ref{mainthm}.

\begin{lem} \label{pain}
Let $\g:=(g_1,\dots,g_m)\in\R[\x]^m$ and let $C$ be a compact subset of $S(\g)$ such that
$g_i$ is strictly \emph{quasi-}concave on $C\cap Z(g_i)$ for each $i\in\{1,\dots,m\}$.
Then there exists a polynomial $h\in\R[T]$ with $h-1$ an sos polynomial such that $h_i:=g_ih(g_i)$ satisfies
\[
\int_0^1(\hess h_i)(u+s(x-u)) \ ds\prec0
\]
for all $i\in \{1,\dots,m\}$, $u\in Z(g_i)$ and $x\in\R^n$ with $\{u+s(x-u)\mid0\le s\le1\}\subseteq C$.
\end{lem}

\begin{proof}
By Lemma \ref{mdfhessian} and the compactness of $C\cap Z(g_i)$, we find $\la>0$ such that
\[F_i:=\la(\nabla g_i)(\nabla g_i)^T-\hess g_i\]
satisfies
\[F_i(x)\succ0\]
for all $i \in \{1,\dots,m\}$ and all $x\in C\cap Z(g_i)$.
The polynomial $h$ will come
out of Lemma \ref{guess} applied to certain values of $R$, $H$, $\ep$ and $\de$, which we will now adjust.
First of all, we choose $R>0$ such that
\[g_i(x)\le R\] for all $i\in\{1,\dots,m\}$ and $x\in C$.
To get $\ep$, we observe that the compact set $C$ is contained in the union of the chain consisting of the open sets
\[\bigcap_{i=1}^m\left(\left\{x\in\R^n\mid g_i(x)>\ep\right\}\cup\left\{x\in\R^n\mid F_i(x)\succ0\right\}\right)\qquad(0<\ep<R)\]
and therefore is contained in those of these sets that belong to a sufficiently small $\ep$, i.e.,
there is $\ep$ with $0<\ep<R$ such that
\[\forall x\in C:\forall i\in\{1,\dots,m\}:\left(g_i(x)\le\ep\implies F_i(x)\succ0\right).\]
By compactness, there exists $\xi>0$ such that
\begin{align}\forall x\in C:\forall i\in\{1,\dots,m\}:\left(g_i(x)\le\ep\implies F_i(x)\succ\xi \id_n\right). \label{matin}
\end{align} 
We choose $\de$ with $0<\de<\ep$ arbitrary and $d>0$ such that
\[\|x-y\|\le d\]
for all $x,y\in C$. The compact subset $C\times C$ of $\R^{2n}$ is contained in
the union of the chain consisting of the open sets
\begin{multline*}
\left\{(x,y)\in\R^n\times\R^n\mid\|x-y\|>\si\right\}\cup\\
\bigcap_{i=1}^m\left(\left\{(x,y)\in\R^n\times\R^n\mid g_i(x)\ne0\right\}\cup\left\{(x,y)\in\R^n\times\R^n
\mid g_i(y)<\de\right\}\right)\\
(0<\si\le d),
\end{multline*}
and therefore is contained in those of these sets that belong to a sufficiently small $\si$, i.e.,
there is $\si$ with $0<\si\le d$ such that 
\begin{align} 
\forall x,y\in C:\left(\|x-y\|\le\si\implies\forall i\in\{1,\dots,m\}:(g_i(x)=0\implies g_i(y)<\de\right)). \label{einzi}
\end{align}
Because $C$ is compact, we can choose $\ta>0$ such that
\[\|F_i(x)\|\le\ta\]
for all $x\in C$ and $i\in\{1,\dots,m\}$. Finally, set
\[H:=\max\left\{\frac{d\ta}{\si\xi},\la\right\}.\]
Choose $h\in \R[T]$ such that $h-1$ is an sos polynomial in $\R[T]$ according to Lemma \ref{guess} and the chosen values of $H$, $R$, $\ep$ and $\de$.
Fix $i\in\{1,\dots,m\}$ and set $h_i:=g_ih(g_i)$. Using the product and chain rule, we calculate
\[\nabla h_i=g_ih'(g_i)\nabla g_i+h(g_i)\nabla g_i=(h(g_i)+g_ih'(g_i))\nabla g_i\]
and therefore
\[\hess h_i=(h(g_i)+g_ih'(g_i))\hess g_i+\nabla g_i\nabla(h(g_i)+g_ih'(g_i))^T.\]
Using
\[\nabla(h(g_i)+g_ih'(g_i))=(2h'(g_i)+g_ih''(g_i))\nabla g_i,\]
it follows that
\[\hess h_i=(h(g_i)+g_ih'(g_i))\hess g_i + (2h'(g_i)+g_ih''(g_i))(\nabla g_i)(\nabla g_i)^T.\]
One now recognizes that conditions (a) and (b) from Lemma \ref{guess} guarantee that
\begin{align*}
\hess h_i(x)&\preceq\left((h(g_i)+g_ih'(g_i))\left(\hess g_i-H(\nabla g_i)(\nabla g_i)^T\right)\right)(x)\\
&\preceq \left(-(h(g_i)+g_ih'(g_i))F_i\right)(x)
\end{align*}
for all $x\in C$ since $H\ge\la$. Now let $u\in Z(g_i)$ and $x\in\R^n$ with \[\{u+s(x-u)\mid0\le s\le1\}\subseteq C.\]
It suffices to show
\[
\int_0^1\left((h(g_i)+g_ih'(g_i))F_i\right)(u+s(x-u))\ ds \succ0.
\]
To this end, we split up the unit interval $[0,1]$ into three disjoint parts
\begin{align*}
J_1&:=\{s\in[0,1]\mid g_i(u+s(x-u))<\de\},\\
J_2&:=\{s\in[0,1]\mid \de\le g_i(u+s(x-u))\le\ep\}\text{ and}\\
J_3&:=\{s\in[0,1]\mid g_i(u+s(x-u))>\ep\}.
\end{align*}
In particular, each $J_k$ is a union of intervals such that $[0,1]=J_1\dotcup J_2\dotcup J_3$. We now
analyze the integral in question on each of these parts separately:
The integral over $J_1$ will contribute a guaranteed amount of positive definiteness, the integral over $J_2$ an unknown amount of
positive semidefiniteness and the integral over $J_3$ will be very small in norm so that it cannot destroy the positive definiteness accumulated
over $J_1$. For further use, we set
\[M:=\max\{h(s)+h'(s)s\mid s\in[\ep,R]\}.\]

\medskip\noindent
\textbf{Analysis on $J_1$.} The subinterval $[0,\frac\si d]$ of $[0,1]$ (note that $\frac\si d\le1$) is contained in $J_1$ since
$\|u-(u+s(x-u))\|=s\|x-u\|\le\frac\si dd=\si$ for $s \in [0,\frac{\sigma}{d}]$ and therefore \[g_i(u+s(x-u))<\de\] for all $s\in[0,\frac\si d]$ by the choice of $\si$ (see
Property \eqref{einzi} above).
By choice of $\xi$, we have that \[F_i(u+s(x-u))\succ\xi\id_n\] for all $s\in J_1$ (in fact also for $s\in J_2$). By Parts (a) and (c) of Lemma \ref{guess},
we have $(h(g_i)+g_ih'(g_i))(u+s(x-u)) > HM$ for all $s\in J_1$. Hence we get with Property \eqref{matin} above that
\[
\int_{J_1}\left((h(g_i)+g_ih'(g_i))F_i\right)(u+s(x-u))\ ds\succ \frac\si dHM\xi\id_n\succeq\frac\si d\frac{d\ta}{\si\xi}M\xi\id_n=\ta M\id_n.
\]

\medskip\noindent
\textbf{Analysis on $J_2$.} We have of course \[F_i(u+s(x-u))\succeq0\] for all $s\in J_2$ (in fact also for $s\in J_1$) and,
by Part (a) of Lemma \ref{guess},
\[(h(g_i)+g_ih'(g_i))(u+s(x-u))\ge0\] for all $s\in[0,1]$.
Hence
\[
\int_{J_2}\left((h(g_i)+g_ih'(g_i))F_i\right)(u+s(x-u))\ ds \succeq0.
\]

\medskip\noindent
\textbf{Analysis on $J_3$.} We have of course $F_i(u+s(x-u))\succeq-\|F_i(u+s(x-u))\|\id_n\succeq-\ta\id_n$ for all $s\in[0,1]$
and therefore
\[
\int_{J_3}\left((h(g_i)+g_ih'(g_i))F_i\right)(u+s(x-u))\ ds\succeq-M\ta\id_n
\]

\medskip\noindent
\textbf{Total analysis.} Finally, we get
\begin{multline*}
\int_0^1\left((h(g_i)+g_ih'(g_i))F_i\right)(u+s(x-u))\ ds\\
\succeq\int_{J_1}\left((h(g_i)+g_ih'(g_i))F_i\right)(u+s(x-u))\ ds\\
+\int_{J_3}\left((h(g_i)+g_ih'(g_i))F_i\right)(u+s(x-u))\ ds\\
\succ\ta M\id_n-M\ta\id_n=0
\end{multline*}
\end{proof}

\begin{thm} \label{linearstability}
Let $\g:=(g_1,\dots,g_m)\in\R[\x]^m$ such that $S(\g)$ is convex with nonempty interior and
$M(\g)$ is Archimedean. Suppose that each $g_i$ is strictly quasiconcave on $S(\g)\cap Z(g_i)$ or
$\g$-sos-concave. Then there is $d\in\N_0$ such that for all $f\in\R[\x]_1$
with $f\ge0$ on $S(\g)$ we have $f\in M_d(\g)$.
\end{thm}

\begin{proof}
Choose $I$ and $J$ such that $\{1,\dots,m\}=I\dotcup J$,
$g_i$ is strictly quasiconcave on $S(\g)\cap Z(g_i)$ for $i\in I$ and $g_j$ is
$\g$-sos-concave for $j\in J$.
Applying Lemma \ref{pain} with $(g_i)_{i \in I}$ instead of $\g$  and the compact subset $C:=S(\g)$ of $S((g_i)_{i \in I})$, we get for each $i\in I$ a polynomial
\[h_i\in M(\g)\]
satisfying $S(h_i)=S(g_i)$, $Z(h_i)=Z(g_i)$ and
\[
\int_0^1(\hess h_i)(u+s(x-u)) \ ds\prec0
\]
for all $u\in S(\g)\cap Z(g_i)$ and $x\in S(\g)$. Setting here $x=u$, we obtain
in particular
\begin{align} \label{hessq}
\hess h_i\prec0\text{ on }S(\g)\cap Z(g_i)
\end{align}
for each $i\in I$.
Set \[h_j:=g_j\] for all $j\in J$. Then \[S(\g)=S(\h).\]
Choose $d_1\in\N_0$ such that \[h_i\in M_{d_1}(\g)\] for all $i\in I\cup J$.
Define for all $i\in I\cup J$ and $u\in\R^n$ a symmetric matrix polynomial $H_{i,u}\in \R[\x]^{n\times n}$ by
\[
H_{i,u}(x)=-\int_0^1\int_0^t(\hess h_i)(u+s(x-u)) \ ds\ dt
\]
for all $x\in\R^n$.
Applying compactness of $S(\g)\cap Z(g_i)$, $S(\g)$ and the unit sphere in $\R^n$ together
with continuity, we find $\de>0$ such that
\[
-\int_0^1(\hess h_i)(u+s(x-u)) \ ds\succeq2\de \id_n
\]
for all $i \in I, u\in S(\g)\cap Z(g_i)$ and $x\in S(\g)$.
For each $t\in[0,1]$, we apply this to $u+t(x-u)\in S(\g)$ instead of $x$ to get
\[
-\int_0^t(\hess h_i)(u+s(x-u)) \ ds = -t\int_0^1(\hess h_i)(u+st(x-u)) \ ds\succeq2t\de \id_n
\]
for all $i \in I, u\in S(\g)\cap Z(g_i)$ and $x\in S(\g)$.
Thus
\[
H_{i,u}(x)\succeq\int_0^12t\de\id_n\ dt=\de\id_n
\]
for all $i\in I$, $u\in S(\g)\cap Z(g_i)$ and $x\in S(\g)$.
Again using the compactness of $S(\g)\cap Z(g_i)$ and continuity, we find some $E>0$
such that
\[\|H_{i,u}\|\le E\]
for all $i\in I$ and $u\in S(\g)\cap Z(g_i)$.
Theorem \ref{putinarmatrixbound} yields $d_2\in\N$ such that
\[H_{i,u}\in M_{d_2}^{n\times n}(\g)\] for all $i\in I$ and $u\in S(\g)\cap Z(g_i)$.
Lemma \ref{bochnercone} yields
$d_3\in\N$ such that
\[H_{j,u}\in M_{d_3}^{n\times n}(\g)\]
for all $j\in J$ and $u\in\R^n$.
For later use, set
\[d_4:=\max\{d_2,d_3\}+2\quad\text{and}\quad d:=\max\{d_1,d_4\}.\]
Now let $f\in\R[\x]_1$ with $f\ge0$ on $S(\g)$. Since $S(\g)$ is nonempty and compact, we
can define $c$ as the minimum of $f$ on $S(\g)$. Exchanging $f$ by $f-c$, we can suppose
without loss of generality that $c=0$.
Then there is some $u\in S(\g)$ with
\[f(u)=0.\] Consider
\[K:=\{i\in I\cup J\mid g_i(u)=0\}
=\{i\in I\cup J\mid h_i(u)=0\}.\]
Because of $\hess h_i(u)\prec0$ (see Property \eqref{hessq}) and continuity, we get a neighborhood $U$ of $u$ such that
\[\hess h_i\prec0\text{ on }U\]
for all $i\in I\cap K$. Since
each $h_j=g_j$ with $j\in J$ is $\g$-sos-concave, we have on the other hand
\[\hess h_j\preceq0\text{ on }S(\g)\]
for all $j\in J$. Combining both, we have in particular that
\[\hess h_k\preceq0\text{ on }S(\g)\cap U\]
for all $k\in K$. Applying Lemma \ref{kkt}, we get a family $(\la_k)_{k\in K}$ of nonnegative Lagrange multipliers such
that $\nabla f=\sum_{k\in K}\la_k\nabla h_k(u)$ (recall that $f$ is linear) and thus
\[\left(f-\sum_{k\in K}\la_kh_k\right)(u)=0\text{ and }\nabla\left(f-\sum_{k\in K}\la_kh_k\right)(u)=0.\]
Fix now $x\in\R^n$. For the map
\[h\colon\R\to \R,\ s \mapsto\left(f-\sum_{k\in K}\la_kh_k\right)(u+s(x-u)),\]
we have $h(0)=0$, $h'(0)=0$ and
\[h''(s)=-\sum_{k\in K}\la_k(x-u)^T((\hess h_k)(u+s(x-u)))(x-u)\]
for $s\in\R$. Hence
\begin{multline*}
\left(f-\sum_{k\in K}\lambda_kh_k\right)(x)=h(1)
\overset{h(0)=0}=\int_0^1 h'(t) \ dt
\overset{h'(0)=0}=\int_0^1 \int_0^t h''(s) \ ds \ dt\\
=\sum_{k\in K}\lambda_k(x-u)^T H_{k,u} (x-u).
\end{multline*}
Since $x\in\R^n$ was arbitrary, we thus have
\[
f-\sum_{k\in K}\lambda_kh_k
=\sum_{k\in K}\lambda_k(\x-u)^T H_{k,u} (\x-u)\in M_{d_4}(\g)
\]
and thus $f\in M_d(\g)$.
\end{proof}

\noindent
Note that it is essential in the previous theorem to require $f$ to be linear.
It is even not enough to require $f$ to be globally convex of small bounded degree
\cite{kl}.

\begin{main} \label{mainthm}
Let $\g:=(g_1,\dots,g_m)\in\R[\x]^m$ such that $S(\g)$ is convex with nonempty interior and
$M(\g)$ is Archimedean. Suppose that each $g_i$ is strictly quasiconcave on $S(\g)\cap Z(g_i)$ or
$\g$-sos-concave. Then $\g$ has an exact Lasserre relaxation.
\end{main}

\begin{proof}
Directly from \ref{linearstability} by the trivial direction of Proposition \ref{duality}.
\end{proof}

\noindent
In the situation of this theorem,
now drop the convexity assumption and consequently ask whether
the \emph{convex hull} of $S(\g)$ (instead of $S(\g)$ itself) equals $S_d(\g)$ for large $d$.
Helton and Nie proved that in this situation the convex hull of $S(\g)$ is
semidefinitely representable \cite[Theorem 4.4]{hn2}. The question arises if it even equals
$S_d(\g)$ for large $d$. This will be proven in our forthcoming paper \cite{ks}
if \emph{all} $g_i$ are strictly quasiconcave on $S(\g)\cap Z(g_i)$. 
However, Example \ref{twodisks} below shows that in this case, one cannot
allow that some of the $g_i$ are linear (or even sos-concave) instead.
To prove this, we need the following
important criterion from \cite[Proposition 4.1]{gn}.

\begin{thm}[Gouveia and Netzer] \label{gnobstr}
Suppose $\g:=(g_1,\dots,g_m)\in\R[\x]^m$, $L \subseteq \R^n$ is a straight line in $\R^n$,
$S(\g) \cap L$ has nonempty interior in $L$ and $u\in S(\g)$ is an element of the boundary of $\overline{\conv(S(\g))} \cap L$ in $L$. Suppose that for each $i$ with $g_i(u)=0$,
$\nabla g_i(u)$ is orthogonal to $L$. Then $S_d(\g)$ \emph{strictly} contains the closure of
the convex hull of $S(\g)$ for all $d$.
\end{thm}

\begin{ex} \label{twodisks}
Let $n:=2$, write $X,Y$ for $X_1,X_2$ and consider $\g:=(g_1,g_2)$ with
\[g_1:=-(1-X^2-Y^2)(4-(X-4)^2-Y^2)\text{ and }g_2:=1-Y.\]
We see that $S(g_1)$ is the disjoint union of two closed disks of different radii. The affine half plane $S(g_2)$
cuts out a piece from the bigger disk and its boundary line $L:=\left(\begin{smallmatrix} 0 \\ 1 \end{smallmatrix}\right) +\R \left(\begin{smallmatrix} 1 \\ 0 \end{smallmatrix}\right)$
is tangent to the smaller disk.
Since $S(g_1)$ is compact, $M(\g)$ is Archimedean by
Theorem \ref{schmu}(b).
By Proposition \ref{qcc}(b), $g_1$ is strictly quasiconcave on $S(\g)\cap Z(g_1)$. The line $L$ is tangent to the smaller disk in the point $\left(\begin{smallmatrix} 0 \\ 1 \end{smallmatrix}\right)$ and passes through the interior of the larger disk.
By the criterion \ref{gnobstr} of Gouveia and Netzer applied with $u:=\left(\begin{smallmatrix} 0 \\ 1 \end{smallmatrix}\right)$,
$S_d(\g)$ strictly contains the convex hull of $S(\g)$ for all $d$.
By inspection of the proof of Gouveia and Netzer, we see more precisely
that each $S_d(\g)$ contains a left neighbourhood of $u$ inside $L$.
\end{ex}

\section*{Acknowledgments}
\noindent
The authors would like to thank all three anonymous referees for their thorough reading that helped to improve the presentation of the material.

\end{document}